\documentclass[12pt]{amsart}
\pdfoutput=1
\usepackage{amssymb,latexsym,vmargin,graphicx,color}

\newtheorem{theorem}{Theorem}

\theoremstyle{definition}

\newcommand{\beql}[1]{\begin{equation}\label{#1}}
\newcommand{\eeq}{\end{equation}}
\newcommand{\comment}[1]{}

\newcommand{\Abs}[1]{{\left|{#1}\right|}}
\newcommand{\Lone}[1]{{\left\|{#1}\right\|_{1}}}
\newcommand{\Linf}[1]{{\left\|{#1}\right\|_\infty}}
\newcommand{\Norm}[1]{{\left\|{#1}\right\|}}
\newcommand{\Mean}{{\bf E}}
\newcommand{\Floor}[1]{{\left\lfloor{#1}\right\rfloor}}

\newcommand{\Prob}[1]{{{\bf{Pr}}\left[{#1}\right]}}
\newcommand{\Set}[1]{{\left\{{#1}\right\}}}

\newcounter{rem}
\newcounter{othm}
\def\theothm{\Alph{othm}} 
\newenvironment{othm}{
  \em
  \vskip 0.10in
  \refstepcounter{othm}
  \noindent{\bf Theorem\ \theothm}
}{\vskip 0.10in}

\begin{document}

\title[Squares of Newman polynomials]{Coefficients of squares of Newman polynomials}

\author[M. Kolountzakis]{Mihail N. Kolountzakis}
\address{Department of Mathematics, University of Crete, Knossos Ave., GR-714 09, Iraklio, Greece}
\email{kolount@gmail.com}

\date{June 2008; revised November 2008}

\thanks{
Supported by research grant No 2569 from the Univ.\ of Crete.
}

\begin{abstract}
We show that there are polynomials $p_N$ of arbitrarily large degree $N$, with coefficients equal to 0 or 1 (Newman polynomials),
such that
$$
\liminf_{N \to \infty} N \Linf{p_N^2} \bigl / p_N^2(1) < 1,
$$
where $\Linf{q}$ denotes the maximum
coefficient of the polynomial $q$ and which, at the same time, are sparse:
$p_N(1)/N \to 0$.
This disproves a conjecture of Yu \cite{yu}. We build on some previous results of Berenhaut and Saidak
\cite{berenhaut-saidak} and Dubickas \cite{dubickas}
whose examples lacked the sparsity. This sparsity we create from these examples by randomization.
\end{abstract}

\hrulefill
\begin{quotation}
\begin{center}
 {\bf Acknowledgement of priority}
\end{center}
\noindent
Results stronger than those contained in this paper, with similar methods, have been obtained by Javier Cilleruelo \cite{cilleruelo} before my paper was written. My paper will not be published. Please do not cite it. 

\ 

Mihalis Kolountzakis
\end{quotation}
\hrulefill

\ 

\maketitle

A {\em Newman polynomial} is a polynomial whose coefficients are 0 or 1. This is a very natural object
and this terminology is naturally not universal across mathematics. Many problems can be expressed using Newman
polynomials and many of them turn out to be quite hard: it is often non-trivial to encode this 0-1 condition
using the data of a specific problem. For instance (drawing from the author's experience) many questions that concern
problems of tiling the integers by translations of finite sets can be expressed using divisibility and factorization
properties of Newman polynomials (see for instance \cite{kolountzakis}).
Similarly such properties of Newman polynomials play a major role
in questions of phase retrieval \cite{lemke-et-al} (how to recover the phase of the Fourier transform of an indicator function of
a finite set of integers if one knows only the modulus of the Fourier transform).
Several extremal problems concerning Newman polynomials are also of interest (see, for instance, the references
in \cite{berenhaut-saidak}).
 
In this note we disprove a conjecture of Yu \cite{yu} which concerns the size of the coefficients
of squares of Newman polynomials.
For a polynomial $p(x)=\sum_{j=0}^d p_j x^j$, with $p_d \neq 0$,
we denote by $\Linf{p}$ the size of the maximal coefficient
in absolute value and by $\Lone{p} = \sum_{j=0}^d \Abs{p_j}$. We also write $\deg{p} = d$.

Write
$$
R(p) = \frac{\Linf{p^2}}{\Norm{p}_1^2}.
$$
For any polynomial $p$ with nonnegative coefficients
we have $\Norm{p}_1^2 = \Norm{p^2}_1$ and, observing that the degree of $p^2$ is $2 \deg{p}$, we obtain easily
\beql{trivial}
R(p) \ge \frac{1}{2\deg{p}+1}.
\eeq
Yu \cite{yu} conjectured that if $p_n$ is a sequence of Newman polynomials
with
\beql{sparsity}
\Norm{p_n}_1 = o(\deg{p_n})
\eeq
then
\beql{yu-conjecture}
\liminf_n R(p_n) \deg{p_n} \ge 1.
\eeq
By \eqref{trivial} the trivial right hand side in \eqref{yu-conjecture} would be 1/2.

That \eqref{yu-conjecture} fails if we ommit the sparsity condition \eqref{sparsity} was shown by
Berenhaut and Saidak \cite{berenhaut-saidak} and by Dubickas \cite{dubickas}.
They exhibited sequences of polynomials
$p_n$ with $\liminf_n R(p_n) \deg{p_n} < 1$ (the liminf was 8/9 in the case of \cite{berenhaut-saidak}
and 5/6 in the case of \cite{dubickas}),
but with $\Norm{p_n}_1 \ge c \deg{p_n}$ for some positive constant $c$.

Our purpose here is to show that the conjecture of Yu mentioned
above fails. We will show that \eqref{yu-conjecture} fails even with the condition \eqref{sparsity}.
For this we will use a sequence of ``dense'' polynomials $p_n$ which satisfies $\liminf_n R(p_n) \deg{p_n} = \rho < 1$
(such as any of those constructed in
\cite{berenhaut-saidak} or \cite{dubickas}) and will construct, for any $\rho < \rho' < 1$, another sequence
of Newman polynomials $q_n$, with $\Norm{q_n}_1 = o(\deg{q_n})$,
which satisfies $\liminf_n R(q_n) \deg{q_n} = \rho' < 1$.

The remainder of this note is devoted to the proof of the following result.
\begin{theorem}
Suppose there exists a sequence of Newman polynomials $p_n$, with degrees tending to infinity,
and positive constants $c_0$, $\rho$ such that
\beql{old}
\Lone{p_n} \ge c_0 \deg{p_n} \mbox{\ \ and\ \ } R(p_n) \le \rho \frac{1}{\deg{p_n}}.
\eeq
Then for every $\rho'>\rho$ there exists an infinite sequence of Newman polynomials $q_n$, with degrees tending to infinity,
such that
\beql{new}
\Lone{q_n} = o(\deg{q_n}) \mbox{\ \ and\ \ } R(q_n) \le \rho' \frac{1}{\deg{q_n}}.
\eeq
\end{theorem}

\begin{proof}
The idea of the proof is to construct the polynomials $q_n$ from the $p_n$ by keeping a random
subset of the monomials in $p_n$. This will achieve the sparsity condition \eqref{sparsity} if we keep the monomials
with small probability. At the same time we are able to control the size of the coefficients of $q_n^2$
using standard tail estimates.
 
Write $N = \deg{p_n}$, assume $N$ is large, and notice 
that our assumptions on $p_n$ and \eqref{trivial} imply that
\beql{height-lb}
\Linf{p_n^2} \ge \frac{c_0^2 N^2}{2N+1} \ge \frac{c_0^2}{3}N.
\eeq
Let $\alpha = \alpha(N) = N^{-1/10}$
and define the random polynomial
$$
q_n(x) = \sum_{j=0}^N q_j x^j
$$
by taking $q_j = \epsilon_j p_j$ with independent $\epsilon_j \in \Set{0,1}$ being equal to 1
with probability $\alpha$ and $0$ with probability $1-\alpha$. Write $(p_n^2)_j$ and $(q_n^2)_j$ for the coefficients of $x^j$ respectively in
the polynomials $p_n^2$ and $q_n^2$.

It follows that $q_n$ is a Newman polynomial of degree at most $N$ and we have
$$
\Norm{q_n}_1 = \sum_{j=0}^N \epsilon_j (p_n)_j,\ \ (q_n^2)_k = \sum_{j=0}^k \epsilon_j \epsilon_{k-j} (p_n)_j (p_n)_{k-j}\ \ (k=0,1,\ldots,2N).
$$
It follows immediately that
$$
\Mean{\Norm{q_n}_1} = \alpha \Norm{p_n}_1
$$
and, if $k$ is odd, we also
have
$$
\Mean{(q_n^2)_k} = \alpha^2 (p_n^2)_k
$$
(the reason for restricting $k$ to be odd is that then the products
$\epsilon_j\epsilon_{k-j}$ that appear are products of independent variables). 

If $k$ is even we have
\begin{eqnarray*}
\Mean{(q_n^2)_k} &=& \alpha^2 (p_n^2)_k + \alpha(1-\alpha)(p_n)_{k/2}^2\\
 &=& \alpha^2 (p_n^2)_k + \theta_k,\ \ \ (0 \le \theta_k < 1).
\end{eqnarray*}
The random variables $\Norm{q_n}_1$ and $(q_n^2)_k$ are both sums of indicator (0-1 valued) random variables.
In the case of $\Norm{q_n}_1$ these random variables are independent while $(q_n^2)_k$ can be
written, depending on whether $k$ is odd or even, as follows.

When $k$ is odd we have
\begin{eqnarray}
(q_n^2)_k &=& \sum_{j=0}^{\Floor{k/2}} \epsilon_j \epsilon_{k-j} (p_n)_j (p_n)_{k-j} +
            \sum_{j=\Floor{k/2}+1}^k \epsilon_j \epsilon_{k-j} (p_n)_j (p_n)_{k-j} \nonumber\\
  &=:& X_{n,k,1} + X_{n,k,2}, \label{odd-split}
\end{eqnarray}
while for even $k$ we have
\begin{eqnarray}
(q_n^2)_k &=& \sum_{j=0}^{k/2-1} \epsilon_j \epsilon_{k-j} (p_n)_j (p_n)_{k-j} +
            \sum_{j=k/2+1}^k \epsilon_j \epsilon_{k-j} (p_n)_j (p_n)_{k-j} + \epsilon_{k/2} (p_n)_{k/2}^2 \nonumber\\
 &=:& Y_{n,k,1} + Y_{n,k,2} + Y_{n,k,3}. \label{even-split}
\end{eqnarray}
The random variables $X_{n,k,1}, X_{n,k,2}, Y_{n,k,1}, Y_{n,k,2}, Y_{n,k,3}$ defined above
are all sums of {\em independent} indicator random variables.
For such random variables we can control the probability of their deviation from their mean
using the following well known result, which we are going to use with $\epsilon$ being a constant
that depends only on $\rho$ and $\rho'$.
\begin{othm}\label{th:chernoff}
{\rm (Chernoff \cite{chernoff}, \cite[Corollary A.1.14]{alon-spencer})}
If $X=X_1+\cdots+X_k$, and the $X_j$ are independent indicator random
variables {\rm (}that is $X_j \in \{0,1\}${\rm )}, then for all $\epsilon>0$
$$
\Prob{\Abs{X-\Mean{X}}>\epsilon \Mean{X}} \le 2 e^{-c_\epsilon \Mean{X}},
$$
where $c_\epsilon>0$ is a function of $\epsilon$ alone
$$
c_\epsilon = \min{\{-\log{(e^\epsilon(1+\epsilon)^{-(1+\epsilon)})},
                        \epsilon^2/2\}}.
$$
\end{othm}
Fixing $\epsilon>0$, our purpose is to avoid the following ``bad events'':
\begin{eqnarray}
E &=& \Set{ \Norm{q_n}_1 < (1-\epsilon) \Mean{\Norm{q_n}_1} } \nonumber\\
 &=& \Set{ \Norm{q_n}_1 < (1-\epsilon) \alpha \Norm{p_n}_1 }, \label{bad-1}
\end{eqnarray}
\beql{bad-2}
E_k = \Set{ (q_n^2)_k > (1+\epsilon) \alpha^2 \Linf{p_n^2} }, \ \ k=0,\ldots,2N,
\eeq
and
\beql{bad-3}
D = \Set{\deg q_n \le \frac{c_0}{2} \deg p_n}.
\eeq
If none of these events holds then
\begin{eqnarray*}
R(q_n) \deg{q_n}  &=& \frac{\Linf{q_n^2}}{\Norm{q_n}_1^2} \deg{q_n}\\
&\le& \frac{1+\epsilon}{(1-\epsilon)^2} \liminf_n R(p_n) \deg{p_n}\\
  &\le& \frac{1+\epsilon}{(1-\epsilon)^2} \rho,
\end{eqnarray*}
which can be made less than $\rho'$ for appropriately small $\epsilon$.
The failure of $E$ and $D$ guarantees the sparseness of $q_n$ since
$$
\Norm{q_n}_1 \le (1-\epsilon) N^{9/10}\ \mbox{ and }\ \deg q_n > \frac{c_0}{2} N.
$$
Therefore it remains to estimate from above the probability that none of the bad events $E$, $E_k$ holds.

Let us start with $\Norm{q_n}_1$. This is a sum of independent indicator random variables with mean
$\Mean{\Norm{q_n}_1} = \alpha \Norm{p_n}_1 \ge c_0 \alpha N = c_0 N^{9/10}$, hence Theorem \ref{th:chernoff} implies
\begin{eqnarray}
\Prob{E} &\le& \Prob{\Abs{\Norm{q_n}_1 - \alpha \Norm{p_n}_1} > \epsilon \alpha \Norm{p_n}_1} \nonumber\\
 &\le& 2 \exp(-c_\epsilon c_0 N^{0.9})\label{estimate-1},
\end{eqnarray}
which tends to 0 with $N \to \infty$. One proves similarly that $\Prob{D} \to 0$.

The summands contributing to the random variables $(q_n^2)_k$ are indicator random variables but there are dependencies.
That's why we need to use the breakups \eqref{odd-split} and \eqref{even-split} above.
The $X$ and $Y$ random variables defined there are sums of independent indicator random variables and we can apply Theorem \ref{th:chernoff}
to them in order to control the probabilities of their deviations from their mean.

Let us deal with the case of odd $k$. The case of even $k$ is treated similarly. We separate the random variables
$X_{n,k,1}, X_{n,k,2}$, $k=0,2,\ldots,2N$, into two groups. In the first group we put those variables whose mean is at most $N^{1/10}$
and in the second group we put the remaining variables. For a given odd $k$ we have the following three cases: (a) both $X_{n,k,1}$
and $X_{n,k,2}$ are in the first group, (b) only one of them is, and (c) none is.
Notice that the mean of an $X$ variable is equal to $\alpha^2=N^{-2/10}$
times the maximum value that variable can take, which corresponds to the case of all relevant $\epsilon_j$ being
equal to 1.

If (a) is the case then the $X$ variables are always at most $N^{3/10}$.
From \eqref{height-lb} it follows that $E_k$ cannot hold.

If (b) is the case then, assuming, without loss of generality,
that the variable $X_{n,k,1}$ is in the first group, we get that always $X_{n,k,1} \le N^{3/10}$ as before.
Therefore, for $E_k$ to hold it must be the case that
$$
X_{n,k,2} \ge (1+\frac{\epsilon}{2})\alpha^2 \Linf{p_n^2} \ge (1+\frac{\epsilon}{2}) \alpha^2 (p_n^2)_k \ge (1+\frac{\epsilon}{2}) \Mean{X_{n,k,2}},
$$
and from Theorem \ref{th:chernoff} we obtain
$$
\Prob{E_k} \le \Prob{X_{n,k,2} > (1+\frac{\epsilon}{2}) \Mean{X_{n,k,2}}} \le 2 \exp(-c_{\epsilon/2} N^{1/10}).
$$
Since the number of relevant $k$ is $O(N)$ this implies that the total probability of the $E_k$ for odd $k$ falling in case (b) tends to 0.

The case (c) is treated similarly to case (b).

We have proved that $\Prob{\cup_k E_k} \to 0$ as $n \to \infty$ by splitting the events into three groups and the proof of the Theorem is complete as this implies that there is a choice of the numbers $\epsilon_j \in \Set{0,1}$ such that the events $E$, $E_k$ do not hold and this implies that the polynomial $q_n$ has the desired properties.
\end{proof}


\end{document}